\documentclass[12pt]{amsart}

\usepackage{geometry} 
\title[Critical $k$-Very Ampleness for Abelian Surfaces]{Critical $k$-Very Ampleness for Abelian Surfaces}
\author{Wafa Alagal}
\email{W.Alagal@sms.ed.ac.uk}
\author{Antony Maciocia}
\email{A.Maciocia@eda.c.uk}
\address{School of Mathematics\\ University of Edinburgh\\ Mayfield
  Road\\Edinburgh\\
EH9 3JZ}
\usepackage[all]{xy}

\begin{document}
\newcommand{\Coh}{\operatorname{Coh}}
\newcommand{\Amp}{\operatorname{Amp}}
\newcommand{\Ext}{\operatorname{Ext}}
\newcommand{\ch}{\operatorname{ch}}
\newcommand{\Supp}{\operatorname{Supp}}
\newcommand{\Hom}{\operatorname{Hom}}
\newcommand{\geg}{\operatorname{deg}}
\newcommand{\tors}{\operatorname{tors}}
\newcommand{\A}{\mathcal{A}}
\newcommand{\MGS}[1]{\mathcal{M}^{\mathrm{GS}}_{#1}}
\newcommand{\MBS}[2]{\mathcal{M}^{\mathrm{BS},#1}_{#2}}
\renewcommand{\P}{\mathcal{P}}

\begin{abstract}
Let $(S,L)$ be a polarized abelian surface of Picard rank one
and let $\phi$ be the function which takes each ample line 
bundle $L'$ to the least integer $k$ such that $L'$ is $k$-very ample
but not $(k+1)$-very ample. We use Bridgeland's
stability conditions and Fourier-Mukai techniques to give a closed
formula for $\phi(L^n)$ as a function of $n$ showing that it is linear
in $n$ for $n>1$. As a byproduct, we calculate the walls in the
Bridgeland stability space for certain Chern characters. 
\end{abstract}
\maketitle

\newtheorem{myle}{Lemma}[section]
\newtheorem{mythe}[myle]{Theorem}
\newtheorem{mypr}[myle]{Proposition}
\newtheorem{mycl}[myle]{claim}
\newtheorem{myco}[myle]{Corollary}
\newtheorem{quotethm}[]{Theorem}
\newtheorem{quoteprop}[]{Proposition}
\theoremstyle{definition}
\newtheorem{myrem}[myle]{Remark}
\newtheorem{myex}[myle]{Example}
\newtheorem{mydef}[myle]{Definition}
\numberwithin{equation}{section}
\section{Introduction}
The notion of $k$-very ampleness was introduced in the 1980s initially
to understand the idea of higher order embeddings. Weaker notions of
$k$-spanned (see \cite{BelS} and \cite{BFS}) and $k$-jet ampleness
(see \cite{BMS}) were
also considered. The definitions given all relate to asking that, for
a variety $V$ and ample line bundle $L$ on $V$, the
natural map $\Gamma(L)\to \Gamma(L/\mathcal{I})$ surjects for certain classes
of sheaf $\mathcal{I}$ of 0-subschemes of $V$. The notions differ in how fat
points are treated. In this paper we only consider the strongest notion
of $k$-very ampleness. Since $k$-very ampleness implies $(k-1)$-very
ampleness it is natural to consider the critical value of $k$ when a
line bundle $L$ is $k$-very ample but not $(k+1)$-very ample. We shall denote
this by $\phi(L)$. In the case when the Neron-Severi group is
generated by a single element (as we shall be assuming), we can view
the function $\phi$ as a function of a positive integer. It is then
natural to hope that $\phi(L^{n+1})$ is related to
$\phi(L^n)$. Unfortunately, $k$-very ampleness is not very 
well behaved with respect to tensoring. It is not even clear that
$\phi(L^n)$ can be expressed as a ``nice'' function of $n$. It is
typically bounded above by a degree $\dim V$ polynomial and below by a
linear function of $n$ and one would expect it to be eventually
polynomial for large enough $n$. We show, in fact, that $\phi(L^n)$
equals $c_1(L)^2(n-1)-2$ for $n\geq2$ for a polarized abelian
surface $(S,L)$ for which $\operatorname{NS}(S)=\langle L\rangle$.

One type of variety where most progress has been made is abelian
varieties. There are a number of, by now, classical results on very
ampleness of line bundles. For example, $L^3$ is always very ample for
any ample line bundle $L$. A little more recently, Debarre, Hulek and
Spandaw showed that a suitably generic $(1,d)$ polarization on a
$g$-dimensional abelian variety is very ample for $d>2^g$. For the
case $g=2$ and Picard rank $1$, this was extended by Bauer and
Szenberg (\cite{BS}) to compute $\phi(1)$ (see Proposition \ref{bauer}
below for the details). The issue of $k$-very ampleness and
$k$-spannedness was also studied by Terakawa who showed that they
coincide for a polarized abelian surface (see \cite[Cor. 4.2]{Ter2}). He also
gave necessary and sufficient conditions for when a line bundle is $k$-very ample
(see \cite{Ter1} and \cite[Theorem 1.1]{Ter2}) but these depend on the
existence of certain divisors and the resulting inequalities are tricky
to solve.

There is a clear relation between $k$-very ampleness and so called
``weak index theorem'' conditions arising in Fourier-Mukai
Theory for abelian varieties. These ideas have been extended by Popa and Pareschi
(\cite{PP1}) who introduce the notion of M-regularity and relate it to
$k$-jet ampleness in \cite{PP}. 

This paper is organized as follows. For the rest of the introduction
we define the $\phi$ function and recall some facts already
established in the literature. We also recall some facts about
Fourier-Mukai transforms and deduce some easy results about $\phi$. In
section 2 we give a brief
introduction to Bridgeland's stability conditions needed
prove the main theorem in this paper. In section 3 we recall the
notion of walls and show that, for the Chern character $(r,l,\chi)$
there are no walls.
We also use general stability machinery to provide a
useful technical lemma (\ref{x<d}) needed to prove our main theorem. In the final
section, we show how to use the technical lemma to bound $\phi$ from
above and then prove that the bound is sharp by computing walls
in the stability space associated to the Chern character
$(1,nl,(n-1)^2d+d+1)$. We then induct on $n$ to deduce the main
theorem making use of our technical lemma again:
\renewcommand{\thequotethm}{\ref{phi_L}}
 \begin{quotethm}
  Let $(S,L)$ be a polarized abelian surface with
  NS$(S)=\langle L\rangle$ and $c_1(L)^2=2d$, then $\phi(L^n) =2(n-1)d-2$. 
\end{quotethm}

\subsection*{Acknowledgements}
The authors are grateful to Kota Yoshioka for pointing out the work of
Hiroyuki Terakawa and to the referee for a number of useful comments
and corrections.

\subsection{$k$-Very Ample}
Let $V$ be a complete algebraic variety of dimension $g$ over an
algebraic closed field $\mathbb{K}$, $X$ a 0-dimensional
subscheme of $V$ with $|X|= d = \dim(H^0(\mathcal{O}_{X}))$ and
$L$ an invertible sheaf on $V$.  

\begin{mydef}
For each 0-scheme $X$ on $V$ we can consider the restriction map
$\rho_{X}$ to $X$ for the space of sections of $L$, which
fits into the exact sequence: 
\[0 \to H^0(V,L\otimes\mathcal{I}_{X}) \to H^0(V,L)
  \overset{\rho_{X}}{\rightarrow} H^0(\mathcal{O}_{X} ) \to
  H^1(V,L\otimes\mathcal{I}_{X})  \to H^1(V,L) \to 0\] 
$L$ is called $k$-very ample if $\rho_{X}$ is surjective for
all purely 0-dimensional subscheme $X$ of length $|X| \leq k+1$.  
\end{mydef}
\begin{myrem}
The following follows easily from the definition
\begin{itemize}
\item $L$ is 0-very ample if and only if  $L$ is generated by global section.
\item $L$ is 1-very ample if and only if  it is very ample.
\item If $L$ is $k$-very ample then $L$ is $(k-1)$-very ample. 
\end{itemize}
\end{myrem}
Let $\Amp(S)$ be the ample cone of $S$. By the properties above there
exists an integer $k$ for all $L\in \Amp(S)$ such that $L$ is $k$-very
ample but not $(k+1)$-very ample.  

\begin{mydef}
Define a map
\begin{equation*}
\phi : \Amp(S) \to \mathbb{Z}_{\geq-1}
\end{equation*}
which takes $L$ into the least integer $k$ such that $L$ is $k$-very
ample but not $(k+1)$-very ample. Define $\phi_L(n):= \phi(L^n)$, and
$\phi(n)$ if $L$ is understood. 
\end{mydef}
There is no obvious reason why this should be a good function of $n$
for any variety and, even for $\mathbb{P}^2$, it is hard to
compute. Specific values for some varieties are, however, well known:
\begin{myex} 
Let $(V,L)$ be a principally polarized abelian variety. Then $\phi_L(2)=0$.
\end{myex}
The following lemma, indirectly proved in \cite{BS} Propositions 3.2
and 3.3, gives the
value of $\phi_L(1)$ and we will reprove it (in slightly greater
generality for arbitrary type $(d_1,d_2)$) in \textsection4 in the spirit
of this paper:
 \renewcommand{\thequoteprop}{\ref{bauer}}
\begin{mypr} \cite{BS}
 If $L$ is an ample line bundle on an
 abelian surface $X$ with NS$(X)=\langle L\rangle$ and $c_1(L)^2=2d$, then 
 \[\phi_L(1)=\left \lfloor \frac{d-3}{2} \right \rfloor.\] 
\end{mypr}
Upper and lower bounds for $\phi$ are also known. It is clear that, if
$H^1(L)=0$ (as is the case, for example, for an ample line bundle on
an abelian variety) then an upper bound for $\phi(n)$ can be given by $\chi(L^n)-1$
since $\chi(L^n\otimes\mathcal{I}_X)=\chi(L^n)-|X|$. For a polarized
abelian surface, this is $n^2d-1$. A more careful analysis (from the
condition $L^2\geq 4k+6$ in \cite[Theorem 1.1]{Ter2}) gives
$\frac{1}{2}(n^2d-3)$ which is consistent
with Theorem \ref{phi_L}. A non-trivial lower bound is much
harder to come by but Reider's Theorem (for the most useful version,
see \cite[\S2]{AB}) provides one, at least when certain divisors do
not exist, as it says that if $c_1(L^n)^2>(k+2)^2$ then $L^n$ is
$k$-very ample. If we apply this to the Picard rank one abelian surface
case where such divisors do not exist, we see that $\phi(n)\geq\lceil
\sqrt{2d}n\rceil-3$. But this is not even sharp for $d=1$ and $n=2$.

Reider's Theorem arose in the situation where $L\otimes\mathcal{I}_Z$ is used
to construct vector bundles of rank $2$. Key in his construction is
the Bogomolov Inequality for semi-stable sheaves. We will also use
this in various places and recall it here (see \cite[Theorem 7.3.1]{HL} for a proof).
\begin{mydef}
A torsion-free sheaf $E$ is $\mu$-stable ($\mu$-semistable) with
respect to $l$ if for each proper subsheaf $F$ we have
$$\mu(F) < \mu(E) (\mu(F) \leq \mu(E))$$ 
where $\mu(E)= c_1(E).l^{g-1}/ r(E)$. 
\end{mydef}
Recall the Bogomolov inequality which provides us with a very useful
necessary condition for $\mu$-semistability:
\begin{mypr} 
 Let $V$ be a smooth projective variety of dimension $n$ and $l$ be an
 ample divisor on $V$. If $E$ is a $\mu$-semistable (with respect to
 $l$) torsion sheaf of
 rank $r$ on $V$, then 
 \begin{equation*}
(r-1)c_1^2(E).l^{n-2}\leq 2rc_2^2(E).l^{n-2}.
\end{equation*}
For the case of an abelian surface this reads:
\begin{equation*}
2r(E)\chi(E) \leq c^2_1(E).
\end{equation*}
\end{mypr}
We will also need to consider a finer stability for sheaves:
\begin{mydef}
A torsion-free sheaf $E$ is Gieseker stable (respectively Gieseker
semistable) with respect to $l$ if for each subsheaf $F$ we have 
$$P(F) < P(E) (P(F) \leq P(E))$$ 
where $P(E)= \dfrac{\chi(E\otimes L^n)}{r(E)}$, is the reduced Hilbert polynomial.  

We let $\MGS{\ch}$ denote the moduli space of Gieseker semistable
sheaves on $S$ with Chern character $\ch$ (or more generally, Simpson
semistable sheaves when the rank is zero). The virtual dimension of
$\MGS{(r,c,\chi)}$ is $2c^2d-2r\chi+2$ and these spaces are non-empty
exactly when this dimension is at least $2$ (in other words, exactly
when the Bogomolov inequality holds). The case when the dimension is
exactly $2$ was proved by Mukai in \cite[Prop 6.22]{Muk} and remaining cases are
dealt with in \cite[Thm 0.1]{Yosh2001}. We will need non-emptiness
specifically for the cases where $c=\pm l$ and $\chi=1$ or
$\chi=2$ which are studied in detail in \cite[\S6]{Yosh2001}.
\end{mydef} 
 
 \subsection{Fourier-Mukai transforms}
Let $V$ and $\hat V$ be smooth projective varieties. Consider the flat projections $V
\overset{\pi}{\leftarrow} V\times \hat{V}
\overset{\hat{\pi}}{\rightarrow} \hat{V}$. 
 Let $\mathcal{P} \in D( V\times \hat{V})$, where $ D( V\times
 \hat{V})$ denotes the derived category of bounded complexes of
 coherent sheaves on $V\times \hat{V}$.  The Fourier-Mukai transform
 $\Phi$ is the functor 
 \begin{equation}
\Phi: D(V) \to D(\hat{V})
\end{equation}
which takes $A$ into $R\pi_*(L\hat{\pi}^*A
\overset{L}{\otimes}\mathcal{P})$ (See \cite{Hu}). Denote its cohomology by
$\Phi^i$. 
In fact, we shall only consider the classical Fourier-Mukai transform
where $\mathcal{P}$ is the Poincar\'e bundle on and abelian surface $V=S$. Then
$\Phi$ has a quasi-inverse given (up to shift) by the transform  
 \begin{equation}
\hat{\Phi}: D(\hat{S}) \to D(S) 
\end{equation}
with kernel $\hat{\mathcal{P}} \in D( \hat{S}\times S)$ , where
$\hat{\mathcal{P}}=s^*\mathcal{P}$ and $s:
S\times \hat{S} \to  \hat{S}\times S$ is $\begin{pmatrix} 
0 &-1 \\ 
 1& 0
\end{pmatrix}$.
\begin{mydef}
An object $E$ satisfies WIT$_n$ if $\Phi^i(E)=0$ for all $i \neq n$.
\end{mydef}
\begin{mydef}
An object $E$ satisfies IT$_n$ if $H^i(E \otimes
\mathcal{P}_{\hat{x}}) = 0$ for all $\hat{x} \in \hat{S}$ the dual of
$S$ and $i \neq n$. In which case, $\Phi^n(E)$ is a locally free
sheaf.  
\end{mydef}
\begin{myex}  \cite{DM}
Any ample line bundle $L$ on an abelian variety is IT$_0$. Any sheaf
which is WIT$_0$ is automatically IT$_0$ by the semi-continuity theorem.
\end{myex}

\begin{mypr}
  If $(S, \Phi )$ is an abelian variety and $L$ is IT$_0$, then $L$ is
  $k$-very ample if and only if $L\otimes\mathcal{I}_{X}$ is WIT$_0$ (and
  hence IT$_0$) for all 0-dimensional subschemes $X$ of length
  $|X| \leq k+1$.  
\end{mypr}
\begin{proof}
$" \Rightarrow "$ Suppose that $L$ is $k$-very ample and
  $L\otimes\mathcal{I}_{X}$ is not WIT$_0$ for some purely 0-dimensional
  subscheme $X$ of length $|X| \leq k+1$, so there exists $\hat{x} \in
  \hat{S}$ such that $H^1(L \mathcal{P}_{\hat{x}} \mathcal{I}_X) \neq
  0$. Pick $x$ for which $\hat{x}= \psi_L(x)$ where $\psi_L: S \to
  \hat{S}$ is the \'etale map which
  takes $x$ into $\tau_x^*L\otimes L^{-1}$, then $H^1(\tau^*_{-x}(L)
  \mathcal{I}_X) \neq 0$ and so
  $H^1(\tau^*_{-x}(L\otimes\mathcal{I}_{\tau_{-x}X})) \neq 0$. 
Hence $H^1(L\otimes\mathcal{I}_{\tau_{-x}X}) \neq 0$ where
$|\tau_{-x}X|\leq k+1$ and this contradicts the assumption.  

" $\Leftarrow$ " Since $L\otimes\mathcal{I}_{X}$ is WIT$_0$ for all purely
0-dimensional subscheme $X$ of length $|X| \leq k+1$, then
$H^1(L\otimes\mathcal{I}_{X})= 0$. Hence $L$ is $k$-very ample by definition.   
\end{proof}
 \begin{mypr}
 Let $(S,L)$ be an irreducible  principally polarized abelian surface,
 then $L^n$ is not $(2n-3)$-very ample. 
\end{mypr}
\begin{proof}
 Let $X$ be a 0-dimensional subscheme of $D_L$ of length $2(n-1)$. Then we have a sequence
 $$0\to L^{n-1}\to L^n\otimes\mathcal{I}_{X}\to Q \to 0$$
 Suppose $Q$ is IT$_0$. The Chern character of $Q$ is
 $\ch(Q)=(0,l,(2n-1)-|X|)$. Since $\hat{Q}$ the transform of $Q$ has
 the Chern character $\ch(\hat{Q})=((2n-1)-|X|, -l,0)=(1,-l,0)$, but
 $\hat{Q}$ is locally-free which is impossible. So Q is not IT$_0$ and
 then $L^n\otimes\mathcal{I}_{X}$ is not IT$_0$. 
 \end{proof}
 Such $X$ we call collinear as $H^0(L\otimes\mathcal{I}_X)\neq 0$ so there
 exists $x \in S$ such that $X\subset \tau_xD_L$, a translation of the
 polarization divisor.  
\begin{myco}
 Let $(S,L)$ be an irreducible principally polarized abelian surface,
 then $\phi(n) \leq 2n-4$ for $n\geq2$.  
\end{myco}


\section{Bridgeland Stabiliy Conditions}\label{BS}
Now we will give a brief review of Bridgeland's stability conditions
(see \cite{TB} and we follow the conventions of \cite{Mwall} and \cite{AB}). Define for any $s\in
\mathbb{R}$ the following  
\begin{equation*}
F_s=\{ E \in \Coh_S | E \text{ is torsion-free and } \mu_+(E) \leq 2ds\},
\end{equation*}
\begin{equation*}
T_s=\{ E \in \Coh_S | E \text{ is torsion or } \mu_-(E/ \tors(E)) >
2ds\},
\end{equation*}
where $\mu_+(E)$ is the slope of the largest slope $\mu$-destabilizing
subsheaf of $E$ and $\mu_-(E)$ is the slope of the lowest slope
$\mu$-destabilizing quotient of $E$. We set

\begin{equation*}
\mathcal{A}_s=\{ A \in D(S) | A^i=0, i\notin \{0,-1\}, H^{-1}(A)\in F_s, H^0(A) \in T_s\}.
\end{equation*}

A group homomorphism $Z_{s,t}$ takes the Chern character $\ch(A)$ into 
\begin{align*}
Z_{s,t}(A) &= \left \langle e^{(s+ti)l}, \ch(A) \right \rangle\\ 
 &= -\chi+2dcs+dr(t^2-s^2)+2tdi(c-rs).
\end{align*}  
For each $A\in \mathcal{A}_s$ the slope $\mu_{s,t}(A)\in\mathbb{Q}\cup\{+\infty\}$ of $A$ is given by:
\begin{align}\mu_{s,t}(A)&=-\frac{\operatorname{Re} (Z_{s,t}(A))}{\operatorname{Im}(Z_{s,t}(A))} \\
&=\frac{\chi-2dcs-dr(t^2-s^2)}{2td(c-rs)}\end{align}  
if $\operatorname{Im}(Z_{s,t}(A))\neq0$ and $+\infty$ otherwise.
\begin{mydef}
We say that $E\in\A_s$ is $\sigma_{s,t}$-stable (respectively,
$\sigma_{s,t}$-semistable) if for all injections $F\to E$ in $\A_s$ we
have
$\mu_{s,t}(F)<\mu_{s,t}(E)$ ($\mu_{s,t}(F)\leq\mu_{s,t}(E)$,
respectively). It is well known that these give sensible stability
conditions on any smooth surface. When $s=0$, we will write $\sigma_t$
for $\sigma_{0,t}$.
\end{mydef}
Then the slope of $E=L^n\otimes\mathcal{I}_{X}$ where $\ch(E) = (1,nl,n^2d-|X|)$ is
$$\mu_{s,t}(E)=\frac{n^2d-|X|-2dns-d(t^2-s^2)}{2td(n-s)}.$$
Note that $n>s$ as $E\in T_s$.

\begin{myrem}\label{destab_is_sheaf}
Now suppose $F\in\A_s$ with $\ch (F)= (r, cl, \chi)$ destabilizes
$L^n\otimes\mathcal{I}_{X}$. Then we have a short exact sequence $F\to E\to Q$
in $\A_s$. Taking cohomology we see that $H^{-1}(F)=0$. Then $F\in
T_s$ and so $c>rs$. Notice also that $H^{-1}(Q)\in F_s$ is
torsion-free and since $0\to H^{-1}(Q)\to F\to E$ is exact, $F$ is
also torsion-free.
\end{myrem}
We also have
$$\mu_{s,t}(F) - \mu_{s,t}(E) > 0.$$
Therefore 
\begin{equation}\label{con}
\frac{\chi-2dcs+dr(s^2-t^2)}{2td(c-rs)}-\frac{n^2d-|X|-2dns-d(t^2-s^2)}{2td(n-s)} > 0.
\end{equation}
Define $f(F,E)$ to be the numerator of (\ref{con}), then 
\begin{align*}
f(F,E)&=(\chi-2dcs+dr(s^2-t^2))(n-s)\\&\hspace{2in}-(n^2d-|X|-2dns-d(t^2-s^2))(c-rs)\\
&=(n-s)\chi-c\bigl(n^2d-|X|-d(s^2+t^2)\bigr)+r\bigl(n^2ds-|X|s-dn(s^2+t^2)\bigr).
\end{align*}
We shall be most interested in the case when $s=0$. Then the destabilizing condition becomes
\begin{equation}
f(F,E)=n\chi-c(n^2d-|X|-dt^2)-dnrt^2> 0. 
\end{equation}
Therefore
\begin{equation*}
n\chi-cn^2d+c|X|> (nr-c)dt^2,
\end{equation*}
and $c\leq nr$ because $\mu(H^{-1}(Q))\leq0$ and
$\mu(F)\leq\mu(F/H^{-1}(Q))\leq\mu(E)$. Hence a necessary condition
for the existence of such a destabilizing object is  
\begin{equation}\label{f2}
n\chi-cn^2d+c|X|> 0.
\end{equation}
Recall from \cite[Prop 14.2]{TB} that in the ``large volume limit''
as $t\to\infty$, the $\sigma_t$-semistable objects $E$ with $\mu(E)>0$
are exactly the Gieseker semistable sheaves (when $s=0$). The case when $\mu(E)<0$
is similar. In this case, $r(E)<0$ when $s=0$:
\begin{mypr}\label{Gstable}
Suppose $F \in \mathcal{A}_0$ with $\mu(F)<0$. Then $F$ is
$\sigma_t$-semistable for all $t\gg0$ if and only if $H^0(F)$ is supported in dimension 0 and
$H^{-1}(F)$ is Gieseker semistable vector bundle.  
\end{mypr}
\begin{proof}
 Proof follows in same way as that of \cite[Prop 14.2]{TB} by observing that if $E$
 is Bridgeland stable for all $t\gg0$ then $H^0(E)$ must be supported
 in dim 0, otherwise $\mu_{0,t}(H^0(E))$ is finite and $H^0(E)$
 destabilizes $E$ for $t\gg0$. Moreover $H^{-1}(E)[1]$ is locally free
 since  
 \begin{equation*}
0\to\mathcal{O}_Z \to H^{-1}(E)[1] \to H^{-1}(E)^{**}[1]\to0
\end{equation*}
is short exact sequence in $\mathcal{A}_0$ and then $\mathcal{O}_Z \to
E$ would destabilize $E$. The fact that $H^{-1}(E)$ is Gieseker semistable
follows in the same way as \cite{TB}. 
\end{proof}
\begin{myrem}\label{duality}
An alternative approach can be seen using an observation of Yanagida
and Yoshioka who show that the Bridgeland stability is preserved 
under $[1]\circ\Delta$, where
$\Delta(E)=\mathbf{R}\mathcal{H}om(E,\mathcal{O}_S)$ (see  \cite[Prop
2.6]{YY}) at least when $c_1\cdot\ell\neq0$. So if $F$ is 
$\sigma_t$-semistable then $F^{\vee}$ is $\sigma_t$-semistable and
$\mu(F^{\vee})> 0$. Then Prop. 14.2 in \cite{TB} implies that
$F^{\vee}$ is Gieseker semistable sheaf. Therefore $F^{\vee\vee} \cong F$ takes
required form. In particular, observe that $H^0(F)\neq0$ exactly when
$F^\vee$ is not locally-free.
\end{myrem}

\begin{myrem}\label{huy}
Huybrechts (\cite{Hupaper}) showed that $\Phi[1]$ preserves $\A_0$ and it
can also be shown (see for example, \cite[Prop 3.2]{MM}) that $E\in\A_0$ is
$\sigma_t$-stable if and only if $\Phi(E)$ is $\sigma_{1/t}$-stable
(and similarly for semistable).
\end{myrem}
\section{Walls and Moduli Spaces}
\begin{mydef}\label{moduli}
We let 
$\MBS{t}{\ch}$ denote the moduli space of $\sigma_t$-semistable
objects in $\A_0$. 
\end{mydef}
It is well known
that these spaces are projective varieties for a wide selection of spaces. For
the case of K3 and abelian surfaces it follows from the sliding down
the wall trick of \cite{Mwall} and applying Remark \ref{huy} but was
also independently observed by Miramide, Yanagida and Yoshioka
\cite[Thm 5.3]{MYY} in the Picard rank 1 case and proved more generally for K3 surfaces in
\cite[Thm 1.3]{bayermacri}. 

For example, in the large volume limit as $t\to\infty$,
$\MBS{t}{\ch}=\MGS{\ch}$ when $c_1(\ch)\cdot l>0$ and $r(\ch)>0$. Equality here
means that the points represent exactly the same objects in
$\operatorname{Coh}(S)\cap\A_0$ up to isomorphism. When the slope is
negative, Remark \ref{duality} implies that the large volume limit of
$\MBS{t}{\ch}$ for large $t$ is given by objects $E^\vee[1]$, where
$[E]\in\MGS{\ch^*:=(ch_0,-ch_1,ch_2)}$ and so
$\MBS{t}{\ch}\cong\MGS{\ch^*}$.

It may happen for some value of $t$ that the two moduli spaces are not
equal. In fact, there will be a strictly decreasing sequence $t_0, t_1,\ldots$ of values
of $t$ on either side of which $\MBS{t}{\ch}$ differ. We call these
walls (sometimes they are called mini-walls when we fix $s$).  Our aim
will be to identify these walls when $\ch=(0,l,\chi)$ and
$\ch=(1,nl,n^2d-|X|)$. In the first case we show there are no walls
which generalizes \cite[Prop 4.2]{Mwall}:
\begin{myle}\label{nowalls}
For $s=0$ there are no walls for $\ch=(r,l,\chi)$ for any $\chi,r \in \mathbb{Z}$.
\end{myle}
\begin{proof}\hspace{-5pt}\footnote{We are grateful to the referee for pointing
    out this elegant proof.}
If $E\in\A_0$ then $c_1(E)\cdot l\geq0$ and $c_1(E)=0$ if and only if
$\mu_{0,t}(E)=\infty$. If $E$ has $c_1(E)=l$ and 
sits in a short exact sequence
\[0\to K\to E\to Q\to 0\]
in $\A_0$ then either \\
(i) $c_1(K)=0$ in which case the sequence destabilizes $E$ for all
$t$, or\\
(ii) $c_1(Q)=0$ in which case the sequence does not destabilize for
any $t$.\\
It follows that there are no walls in $\A_0$.
 \end{proof}

\begin{myrem}\label{modequal}
It follows that
$\MGS{(0,l,\chi)}=\MBS{t}{(0,l,\chi)}$ for all $t$ and $s=0$
 and, by Remark \ref{huy}, $\ch=(\chi,-l,0)$ also has no walls for $s=0$. Hence, for all $t>0$,
 \begin{equation*}
 \mathcal{M}^{BS,t}_{(\chi, l,0)}=
 \mathcal{M}^{GS}_{(\chi, l,0)},\quad\text{when $\chi\geq0$}
\end{equation*}
and
\[
 \mathcal{M}^{BS,t}_{(\chi,l,0)}= \Delta
 \mathcal{M}^{GS}_{(-\chi,l,0)}[1],\quad\text{when $\chi<0$}.
\]
\end{myrem}

\begin{mydef}
 We shall say that the moduli space $\mathcal{M}_{(r,cl,\chi)}^{BS,t}$
 of Bridgeland stable sheaves of Chern character $(r,cl,\chi)$
 satisfies IT$_0$ (respectively WIT$_0$) if and only if for each $E$
 representing an object of $\mathcal{M}_{(r,cl,\chi)}^{BS}$, $E$ satisfies
 IT$_0$ (respectively WIT$_0$).
\end{mydef}
For example  $\mathcal{M}_{(1,nl,n^2d-k)}^{BS,t}$ is IT$_0$ for all $t$
if and only if $L^n$ is $(k-1)$-very ample and so $\phi_L(n)\geq k-1.$ 
 Note that if $\mathcal{M}$ is a fine
 moduli space and $[E]\in\mathcal{M}$ then $E$ is IT$_0$ if and only if all
 $F\in[E]\in\mathcal{M}$ are IT$_0$. This may not be true when
 the moduli space is not fine (and there exist non-Gieseker stable
 sheaves) because the IT$_0$ condition is not
 preserved by $S$-equivalence. However, the moduli spaces we consider below
 will all be fine.

The following technical result will be useful in the next section:
\begin{myle}\label{x<d}
$\MGS{(0,l,\chi)}$ is IT$_0$ if and only if $\chi\geq d+1$.
\end{myle}
\begin{proof}
We use Proposition \ref{Gstable}, Remarks \ref{duality} and
\ref{modequal}, and Lemma \ref{nowalls} to give isomorphisms
\[\MGS{(0,l,\chi)}\overset{\Phi[1]}\cong
\MBS{t}{-(\chi,-l,0)}\overset{[1]\Delta}{\cong}
\MBS{t}{(\chi,l,0)}=\MGS{(\chi,l,0)}\]
for all $t>0$.
Then $[E]\in\MGS{(0,l,\chi)}$ is IT$_0$ if and only of
$[\Phi(E)[1]]\in\MBS{t}{-(\chi,-l,0)}\cap\MGS{(\chi,-l,0)}[1]$ which
holds if and only if $\Delta\Phi(E)\in\MGS{(\chi,l,0)}$ is
locally-free.
But, since all representative sheaves of $\MGS{(\chi,l,0)}$ must be
$\mu$-stable, we see that there are non-locally-free sheaves in
$\MGS{(\chi,l,0)}$ if and only if
$\MGS{(\chi,l,1)}\neq\emptyset$. This happens exactly when the
Bogomolov inequality holds for the Chern character $(\chi,l,1)$, in
other words when $\chi\leq d$ as required.
\end{proof}

\section{Polarization Line Bundles}\label{general}

 Let $L$ be polarization line bundle on an abelian surface $S$ such
 that NS$(S)=\langle L\rangle$. Let
 $c_1(L)=l$ and $l^2=2d$. In this section we will prove some
 lemmas that help us to find the value of $\phi_L(n)$. 
We start with the case $n=1$: 
 \begin{mypr}\label{bauer}
 If $L$ is an ample line bundle with $c_1(L)^2=2d$ on an abelian surface $X$, then 
 \[\phi_L(1)=\left \lfloor \frac{d-3}{2} \right \rfloor\]
\end{mypr}
\begin{proof}
 The Chern character of $E=L\otimes\mathcal{I}_{X}$ is
 $(1,l,d-|X|)$.  Then
 $\ch(\Phi(E))=(d-|X|,-l,1)$. Such objects are all locally-free
 sheaves exactly when there are no stable sheaves with Chern character
 $(d-|X|,-l,2)$.  These exist exactly when the Bogomolov inequality
 holds for such a Chern character (see Definition \ref{moduli}). This gives us the criterion
 $2(d-|X|)\leq d$, so $|X|\geq d/2$. Hence
 $\MGS{(1,l,d-|X|)}$ is IT$_0$ if and only if $|X| \leq \left
 \lfloor \dfrac{d-1}{2} \right \rfloor$. Then $\phi_L(1)=\left \lfloor
 \dfrac{d-3}{2} \right \rfloor$.
\end{proof}
The $n=1$ case is exceptional and we now assume $n>1$ and find an
upper bound for $\phi_L(n)$:
 \begin{mypr}\label{uper}
  Let $(S,L)$ be a polarized abelian surface such that
  $\operatorname{NS}(S)=\langle L\rangle$ and $c_1(L)^2=2d$, then
  $\phi_L(n) \leq2(n-1)d-2$ for $n>1.$  
\end{mypr}
 
\begin{proof}
 By Lemma \ref{x<d}, there is $Q$ with Chern character
 $\ch(Q)=(0,l,d)$ which is not IT$_0$. Since $\chi(L^{-n+1}\otimes
 Q)=d(3-2n)<0$ for $n>1$ we have $\Ext^1 (Q, L^{n-1}) \neq 0$. Pick a non
 trivial extension  
\[0\to L^{n-1}\to E\to Q\to 0\]
and suppose $T \hookrightarrow E$ is its torsion subsheaf. Then we
have the following diagram:

\begin{displaymath}
    \xymatrix{ &&F \ar[r] & Q/T & \\
               0 \ar[r] & L^{n-1} \ar[r]\ar[ur] &E \ar[r]\ar[u] & Q \ar[r]\ar[u]& 0 \\
               &&T\ar[u] \ar@{=}[r] & T\ar[u]  & }
\end{displaymath}
Then $Q/T$ must be supported in dimension zero. But then
$\Ext^1(Q/T,L^{n-1})=0$ and so $L^{n-1}\to F\to Q/T$ must split, which
is impossible as $F$ is torsion-free and $Q/T$ is torsion. Hence $T=0$.
 Then $E\cong L^n\otimes\mathcal{I}_X$ for some $X$ of length $|X|=2d(n-1)$ and $E$ is not IT$_0$. 
\end{proof}
The following theorem proves that the upper pound of $\phi_L(n)$ in Proposition \ref{uper} is sharp.
 \begin{mythe}\label{phi_L}
  Let $(S,L)$ be a polarized abelian surface with
  $\operatorname{NS}(S)=\langle L\rangle$ and $c_1(L)^2=2d$, then $\phi(L^n) =2d(n-1)-2$.
\end{mythe}
\begin{proof}
 By Proposition \ref{uper} we need to show that $\phi_L(n)
 \geq2(n-1)d-2$, and we do this by showing that
 $\MBS{t}{(1,nl,n^2d-k)}$ is IT$_0$ for all $t$ and
 $k=2d(n-1)-1$. Suppose that $E\cong L^n\otimes\mathcal{I}_X$ where
 $|X|=2d(n-1)-1$ is not IT$_0$ and $\hat E:=\Phi(E)$ is $\sigma_t$-stable for
 all $t\gg0$. Then $\hat{E}$ is a two-step
 complex such that $H^{-1}(\hat{E})$ is Gieseker stable and $H^0(\hat{E})$
 is in the form $\mathcal{O}_Z$, by Proposition \ref{Gstable}. The
 Chern character $\ch(H^{-1}(\hat{E}))=((n-1)^2d+d+1,-n,1+|Z|)$. By
 Bogomolov 
\begin{equation}
\bigl((n-1)^2d+d+1\bigr)(1+|Z|)\leq n^2d.
\end{equation}
Therefore $|Z|\leq \dfrac{2d(n-1)-1}{dn^2-2d(n-1)+1}$. But
$d(n-2)^2+2>0$ and so 
\[dn^2-2d(n-1)+1>2d(n-1)-1.\] 
Hence, $|Z|<1$. Therefore
$H^0(\hat{E})=0$ and so $E$ is IT$_0$. 
If $E$ is $\sigma_t$-stable for all $t$, then it follows that
$\Phi(E)$ is $\sigma_t$-stable for all $t$ (and so also for
$t\gg0$). This happens when there are no walls. Unfortunately, there
are walls in general. To finish the proof we will identify all the
walls and show that all $\sigma_t$-semistable objects are IT$_0$ directly.  

\begin{myle}
If $e\in \mathcal{A}_0$ destabilizes $L^n\otimes\mathcal{I}_X$ with $|X|=
2d(n-1)-1$, then $e$ is a rank $1$ torsion-free sheaf.
\end{myle}
\begin{proof}
By Remark \ref{destab_is_sheaf}, $H^{-1}(e)=0$ and $E:=H^0(e)\cong e$ is
torsion-free. Suppose $\ch(E)=(r,g'l,\chi)$ and let $q=L^n\otimes\mathcal{I}_X/E$ in $\A_0$. Then
we have a long exact sequence in $\Coh(S)$:
\begin{equation}
0\to H^{-1}(q) \to E \to L^n\otimes\mathcal{I}_X \to H^0(q)\to 0
\end{equation}
Since $H^{-1}(q) \in \mathcal{F}_0$ and $E \in \mathcal{A}_0$, then
$\mu(E)\geq 0 \geq \mu(H^{-1}(Q))$. Then there is an integer $g>0$
such that $c_1(E)=(nr-g)\ell$. Then $c_1(H^{-1}(Q))=
(nr-g-n+m)\ell\leq 0$ where $m\geq 0$. Therefore
$0 < nr-g\leq n-m\leq n$. Hence, $c_1(E)$ can be written as $c_1(E)=(n-c)\ell$ for
some positive integer $c<n$. Since we can assume $E$ is Bridgeland
stable it must be simple and so the Bogomolov inequality holds (as
this is just the statement that the moduli space of simple
torsion-free sheaves has dimension at least 2) and so we can write 
\begin{equation*}
\chi(E) = \frac{(n-c)^2d}{r}-k,
\end{equation*}
for some rational number $k\geq0$.
Since $E$ is a destabilizer of $L^n\otimes\mathcal{I}_X$, we have
$f(E,L^n\otimes\mathcal{I}_X) >0$. Therefore from a condition (\ref{f2}), we
get 
\begin{equation}\label{n-c}
 \bigl((n-c)^2d -kr\bigr)n-(n^2d-2dn+2d+1)(n-c)r>0,
\end{equation}
Rearrange (\ref{n-c}), we obtain
\begin{equation}\label{n-c2}
(n-c)\bigl(-(n-1)^2dr-dr-r+dn^2-cdn\bigr)>krn>0
\end{equation}
As $n-c>0$, then we get walls if $-(n-1)^2dr-dr-r+dn^2-cdn>0$ so 
\begin{equation*}
\frac{1}{r}>\left(1-\frac{1}{n}\right)^2+\frac{1}{n^2}+\frac{1}{dn^2}\geq\frac{1}{2},
\end{equation*}
for all $n$, since $d>0$. Hence $r=1$. \end{proof}
\begin{myrem}
 The previous lemma proved that the Chern character of any
 destabilizer of $L^n\otimes\mathcal{I}_X$ is given by $\ch(E)=\left(1,n-c,
(n-c)^2d-k\right)$ which means that $E$ is in the form
 $L^{n-c}\otimes\P_{\hat x}\otimes\mathcal{I}_Y$, for some $\hat x\in\hat S$ and $|Y|=k$. 
\end{myrem}
\begin{myle}\label{n-m}
 If $L^{n-m}\otimes\P_{\hat x}\otimes\mathcal{I}_Y$ destabilizes
 $L^n\otimes\mathcal{I}_X$ where $|X|=2d(n-1)-1$, then $m=1$. 
\end{myle}
\begin{proof} We assume, without loss of generality, that $\hat x=0$.
Suppose that $F=L^{n-m}\otimes\mathcal{I}_Y$ with $|Y|=k$ destabilizes
$E=L^n\otimes\mathcal{I}_X$, then $\mu(F)-\mu(E)\geq0$. But if
$\mu(E)=\mu(F)$ then $\mu_{0,t}(E/F)=\infty$ and so $F$ does not destabilize. Therefore from a
condition  (\ref{n-c2}), we get 
\begin{equation*}
(n-m)(-(n-1)^2dr-dr-r+dn^2-dnm)>krn.
\end{equation*}
Then we get walls if and only if 
\begin{equation}\label{fn-m}
\left(1-\frac{m}{n}\right)(2dn-dnm-2d-1)>k\geq0. 
\end{equation}
Since $1-\dfrac{m}{n}$ is positive, this happens if and only if
$2dn-dnm-2d>1$ and then $2\geq 2-\dfrac{d+1}{nd}>m>0$. Hence, $m=1$. 
\end{proof}
\begin{myle}
 If $L^{n-1}\otimes\P_{\hat x}\otimes\mathcal{I}_Y$ destabilizes $L^n\otimes\mathcal{I}_X$ for some
 $X$ where $|X|=2d(n-1)-1$, then $|Y|<d(n-2)-1\leq 2d(n-2)-1$. 
\end{myle}
\begin{proof} Without loss of generality we assume $\hat x=0$.
 Take $F,E$ as Lemma \ref{n-m}, then from \ref{fn-m} we get:
 \begin{equation}
\left(1-\frac{1}{n}\right)(dn-2d-1)>|Y|\geq0
\end{equation}
Since $0<1-\dfrac{1}{n}<1$, we have $dn-2d-1>|Y|$.
\end{proof}
To complete proof of Theorem \ref{phi_L}, we now induct on
$n\geq2$. If $n=2$, then $d(n-2)=0$ and so there are no walls, which
establishes the result for the case $n=2$.
 
 Suppose that the statement is true for
 $n-1\geq2$. i.e. $L^{n-1}\otimes\mathcal{I}_X$ is IT$_0$ for all $X$ with
 $|X|=2d(n-2)-1$  To prove that $L^{n}\otimes\mathcal{I}_{X}$ is IT$_0$ for
 all $X$ with $|X|=2d(n-1)-1$ , we know that the only possible walls
 are given by $L^{n-1}\otimes\mathcal{I}_Y$ where $|Y|<2d\bigl((n-1)-1\bigr)-1$. Then
 there is a short exact sequence 
 \begin{equation*}
0\to L^{n-1}\otimes\mathcal{I}_Y\to L^{n}\otimes\mathcal{I}_{X} \to Q\to0
\end{equation*}
By induction, $L^{n-1}\otimes\mathcal{I}_Y$ is IT$_0$ and by Lemma \ref{x<d},
$Q$ is IT$_0$ as well, since 
\[\chi(Q) =-(n-1)^2d+|Y|+n^2d-2d(n-1)+1= 1+d+|Y|\geq d+1.\]
Hence
$L^{n}\otimes\mathcal{I}_{X}$ is IT$_0$ for all $X$ with $|X|=2d(n-1)-1$. 
\end{proof}


 \bibliographystyle{alpha}
 \bibliography{p58}{}

\end{document}